\documentclass[15pt]{amsart}
\usepackage{amscd, amsmath, amsthm, amssymb, MnSymbol}
\usepackage[all]{xy}

\newcommand\Bs{{\mathrm{Bs}}}
\newcommand\Cl{{\mathrm{Cl}}}
\newcommand\Pic{{\mathrm{Pic}}}
\newcommand\Eff{{\mathrm{Eff}}}
\newcommand\Exc{{\mathrm{Exc}}}
\newcommand\Nef{{\mathrm{Nef}}}
\newcommand\Mov{{\mathrm{Mov}}}
\newcommand\Cox{{\mathrm{Cox}}}
\newcommand\SAmp{{\mathrm{SAmp}}}
\newcommand\Aut{{\mathrm{Aut}}}

\newtheorem{theorem}{Theorem}[section]
\newtheorem{lemma}[theorem]{Lemma}
\newtheorem{proposition}[theorem]{Proposition}
\newtheorem{conjecture}[theorem]{Conjecture}
\newtheorem{question}[theorem]{Question}
\newtheorem{problem}[theorem]{Problem}
\newtheorem{corollary}[theorem]{Corollary}
\theoremstyle{definition}     
\newtheorem{definition}[theorem]{Definition}

\theoremstyle{remark}
\newtheorem{remark}[theorem]{Remark}
\numberwithin{equation}{section}

\begin{document}

\date{May 5, 2018}

\title[Examples of Mori dream surfaces of general type with $p_g=0$]{Examples of Mori dream surfaces \\ of general type with $p_g=0$}

\author[J. Keum]{JongHae Keum}
\address{School of Mathematics, Korea Institute for Advanced Study, Seoul 02455, Republic of Korea} \email{jhkeum@kias.re.kr}

\author[K.-S. Lee]{Kyoung-Seog Lee}
\address{Center for Geometry and Physics, Institute for Basic Science (IBS), Pohang 37673, Republic of Korea} \email{kyoungseog02@gmail.com}

\thanks{
 The present work was supported by Institute for Basic Science (IBS-R003-Y1) and National Research Foundation of Korea (NRF)}


\begin{abstract}
In this paper we study effective, nef and semiample cones of minimal surfaces of general type with $p_g=0.$ We provide examples of minimal surfaces of general type with $p_g=0, 2 \leq K^2 \leq 9$ which are Mori dream spaces.  On these examples we also give explicit description of effective cones and all irreducible reduced curves of negative self-intersection. We also present non-minimal surfaces of general type with $p_g=0$ that are not Mori dream surfaces.
\end{abstract}
\maketitle


\section{Introduction}

Birational geometry of a variety constructed by GIT quotient is closely related to the variation of GIT. Cox rings are invariants playing an important role in this interaction. It turns out that Cox rings contain geometric, birational and arithmetic information of algebraic varieties.
It is well-known that if the Cox ring of an algebraic variety is finitely generated then the variety enjoy ideal properties in the aspect of minimal model program. Such varieties are called Mori dream spaces (cf. \cite{HK}).

Recently, Cox rings of various varieties have been studied by many authors, for example, surfaces, moduli spaces of rational curves with marked points, wonderful varieties and log Fano varieties.  See \cite{ADHL} and references therein for more detail. However it seems that not much is known about Cox rings of varieties of general type. Moreover there are not many literatures discussing Mori dream spaces of general type. Therefore it is interesting to find more examples of Mori dream spaces which are varieties of general type and to investigate their properties.

In general, it is hard to compute the effective, nef or semiample cones of a given variety of general type and hence it is difficult to determine whether it is a Mori dream space or not. Therefore we will focus on special classes of algebraic varieties of general type in this paper. Our candidates of Mori dream spaces of general type are minimal surfaces of general type with $p_g=0$. These surfaces have been studied for a long time and there are many works about them (see \cite{BCP} for a survey). Moreover, it turns out that there are some similarities between del Pezzo surfaces and minimal surfaces of general type with $p_g=0.$ It is well-known that del Pezzo surfaces are Mori dream spaces. Therefore it is natural to ask the following questions.

\begin{question}
Let $X$ be a minimal surface of general type with $p_g=0.$ When the effective cone of $X$ is rational polyhedral cone? When the nef cone and semiample cones of $X$ are the same? When the Cox ring of $X$ is finitely generated?
\end{question}

In this paper we provide examples of Mori dream surfaces which are minimal surfaces of general type with $p_g=0$ and $2 \leq K^2 \leq 9.$ We also compute their effective cones explicitly. It is easy to provide examples and compute their effective cones when $K^2$ is large but this task becomes more difficult when $K^2$ gets smaller. Indeed, we use specific structures of our examples to compute the effective cones and to prove that they are Mori dream surfaces.

\begin{theorem}\label{Main}
The following minimal surfaces of general type with $p_g=0$ are Mori dream surfaces and their effective cones are computed explicitly $($see the subsection in each case$)$.
\begin{enumerate}
\item Fake projective planes $($all have $K^2=9)$
\item Surfaces isogenous to a higher product of unmixed type $($all have $K^2=8)$
\item Inoue surfaces with $K^2=7$
\item Surfaces with $K^2=7$ constructed by Y. Chen
\item Kulikov surfaces with $K^2=6$
\item Burniat surfaces with $2 \leq K^2 \leq 6$
\item Product-quotient surfaces with $K^2=6, G=D_4 \times \mathbb{Z}/2\mathbb{Z}$
\item A family of Keum-Naie surfaces which are product-quotient surfaces with $K^2=4, G=\mathbb{Z}/4\mathbb{Z} \times \mathbb{Z}/2\mathbb{Z}$
\end{enumerate}
\end{theorem}

In particular, there is a minimal surface of general type with $p_g=0$ which is a Mori dream space for any $2 \leq K^2 \leq 9.$
On the other hand, we do not know an example of minimal surface of general type with $p_g=0$ which is not a Mori dream space. It would be an interesting task to determine which minimal surfaces of general type with $p_g=0$ are Mori dream surfaces.

\begin{problem}
Classify Mori dream spaces among minimal surfaces of general type with $p_g=0$.
\end{problem}

We believe that an answer to this problem could be an important step toward classification of surfaces of general type with $p_g=0$, which has remained as one of the most important and difficult problems in algebraic geometry.

It is easy to see that for any Mori dream surface $X$ the bounded negativity conjecture holds, that is, there is a nonnegative number $b_X$ such that $C^2 \geq -b_X$ for any irreducible reduced curve $C$. (For historical background of the conjecture we refer to \cite{BHKKMSRS}.)

For the examples in Theorem \ref{Main}, by refining the explicit computation of effective cones we are able to compute all negative curves explicitly in each case. Here a negative curve means an irreducible reduced curve with negative self-intersection.

\begin{theorem}\label{main}[Negative Curves]
 For each surface in Theorem \ref{Main} all negative curves are computed explicitly. In each case $(1)-(8)$ the number of negative curves and the list of the pairs $(C^2, p_a(C))$ for negative curves $C$ are given as follows. Here  $m(C^2, p_a(C))$ means $m$ copies of $(C^2, p_a(C)).$
\begin{enumerate}
\item None
\item None
\item $3: 2(-1,1), (-1,2)$
\item $4: (-1,1), (-1,2), (-1,3), (-4,2)$
\item $6: 6(-1,1)$
\item Burniat surfaces with $2 \leq K^2 \leq 6$
\begin{enumerate}
\item $6 : 6(-1,1)$ if $K^2 = 6;$
\item $10 :  9(-1,1), (-4,0)$ if $K^2 = 5;$
\item $16 : 12(-1,1), 4(-4,0)$ if non-nodal with $K^2 = 4;$
\item $13: 10(-1,1), 2(-4,0), (-2,0)$ if nodal with $K^2 = 4;$
\item $15 : 9(-1,1), 3(-4,0), 3(-2,0)$ if $K^2 = 3;$
\item $16: 6(-1,1), 6(-2,0), 4(-4,0)$  if $K^2 = 2;$
\end{enumerate}
\item $4: 2(-2, 0), (-1,1), (-1, 2)$
\item $8: 4(-1, 1), 4(-2, 0)$
\end{enumerate}
\end{theorem}

Finally we present non-minimal surfaces of general type with $p_g=0$ that are not Mori dream surfaces. The first named author and Naie (\cite{Keum}, \cite{Naie}) constructed a family of minimal surfaces of general type with $p_g=0$ and $K^2=4$ as double covers of an Enriques surface with eight nodes. The branch locus is the disjoint union of the eight nodes and a curve $C$ with $C^2=8$. The curve $C$ may have simple singularities and is linearly equivalent to $F_1+F_2$ with $F_1F_2=4$, where each $|F_i|$ defines an elliptic fibration.  These surfaces, called Keum-Naie surfaces, were discovered by the first named author and later investigated by Naie. If $C$ has simple singularities, then the double cover is a Keum-Naie surface with rational double points. Bauer and Catanese \cite{BC2} proved that the connected component of the Gieseker moduli space corresponding to Keum-Naie surfaces is irreducible, normal, unirational of dimension 6.   Indeed, Enriques surfaces with such a configuration of curves form a 2-dimensional moduli and such a curve $C$ moves in a 4-dimensional linear system.

\bigskip

Let $S$ be an Enriques surface with a disjoint union of nine curves that are eight $(-2)$-curves and a curve $C$, as described as above. Let
$$Y \to S$$ be the double cover of $S$ branched along the nine curves. The surface $Y$ has eight $(-1)$-curves lying over the eight $(-2)$-curves, thus $Y$ is a blow up of a Keum-Naie surface at suitably chosen eight points. By the classification of Kond\=o  \cite{Kon}  Enriques surfaces with finite automorphism groups form a 1-dimensional family. Therefore, if $S$ is general, then $\Aut(S)$ is infinite, hence  $S$ is not a Mori dream space. It follows that $Y$ is not a Mori dream space.

\begin{theorem}\label{notMori}
Let $S$ be an Enriques surface with a disjoint union of nine curves that are eight $(-2)$-curves and a curve $C$ with  simple singularities, as described as above. Let
$Y$ be the double cover of $S$ branched along the nine curves. If $S$ is general, then $Y$ is a non-minimal surface of general type with $p_g=0$ and $K_Y^2=-4$  that is not a Mori dream surface.
\end{theorem}

\bigskip

{\bf Notations.} We will work over $\mathbb{C}.$ When $G$ is an abelian group, then $G_{\mathbb{R}}$(resp. $G_{\mathbb{Q}}$) will denote $G \otimes_{\mathbb{Z}} {\mathbb{R}}$(resp. $G \otimes_{\mathbb{Z}} {\mathbb{Q}}$). A variety will mean a normal projective variety. If $X$ be a normal $\mathbb{Q}$-factorial variety, then we will use the following notations. \\
$K_X$ : the canonical divisor on $X.$ \\
$\Cl(X)$ : divisor class group of $X.$ \\
$\Pic(X)$ : Picard group of $X.$ \\
$\rho(X)$ : Picard number of $X.$ \\
$\Eff(X)$ : the effective cone of $X$. \\
$\Nef(X)$ : the nef cone of $X$. \\
$\Mov(X)$ : the movable cone of $X$. \\
$\SAmp(X)$ : the semiample cone of $X$. \\
Let $D_1, D_2$ are two divisors on $X.$ We write $D_1 \sim D_2$(resp. $D_1 \sim_{num} D_2$) to denote that they are linearly(resp. numerically) equivalent. \\

{\bf Acknowledgements.} The second named author thanks Ingrid Bauer, Fabrizio Catanese, Sung Rak Choi, June Huh, DongSeon Hwang, Jihun Park, Jinhyung Park, Yongjoo Shin and Joonyeong Won for helpful conversations and discussions. Part of this work was done when he was a research fellow of KIAS.

\section{Preliminaries}

In this section we recall basic definitions and results about effective, nef and semiample cones of algebraic surfaces, Mori dream spaces, especially Mori dream surfaces. We also prove a useful criterion to provide many new examples of Mori dream surfaces.

\subsection{Effective, nef and semiample cones of surfaces}

Effective, nef, movable and semiample cones of algebraic varieties are key tools of birational geometry (cf. \cite{KMM,KM}).

\begin{definition}
Let $X$ be a normal projective variety and $D$ be a Weil divisor on $X.$ We will use $\Bs|D|$ to denote the base locus of $|D|.$ \\
(1) The stable base locus is the intersection of all base locus of multiples of $D,$ i.e.
$$ \mathbb{B} |D|:= \bigcap_{k \in \mathbb{Z}_{\geq 1}} \Bs|kD|. $$
(2) The effective cone $\Eff(X)$ is the convex cone generated by effective divisors. We will use $\overline{\Eff(X)}$ to denote the closure of $\Eff(X)$ in $\Cl(X)_{\mathbb{R}}.$ \\
(3) The nef cone $\Nef(X)$ is the convex cone generated by nef divisors. \\
(4) A Weil divisor $D$ is movable if $\mathbb{B} |D|$ has codimension at least 2. The moving cone $\Mov(X)$ is the convex cone generated by movable divisors. \\
(5) A Weil divisor $D$ is semiample if $\mathbb{B} |D|$ is empty. The semiample cone $\SAmp(X)$ is the convex cone generated by semiample divisors.
\end{definition}

Let $X$ be a smooth projective surface with $q=0.$ Then the Picard group $\Pic(X)$ is a finitely generated abelian group and $\Pic(X)_{\mathbb{R}}$ is a finite dimensional vector space.

\begin{proposition}\cite{AL}
Let $X$ be a smooth projective surface with $q=0.$ Then we have the following inclusions.
$$ \SAmp(X) \subset \Mov(X) \subset \Nef(X) \subset \overline{\Eff(X)}  $$
\end{proposition}

In order to check whether a given surface is Mori dream or not, the first thing to do is to determine whether the effective cone of the surface is a rational polyhedral cone or not. Sometimes we can compute the effective cone of a surface explicitly. Let us recall a helpful proposition of Artebani and Laface.

\begin{proposition}\cite[Proposition 1.1]{AL}
Let $X$ be a smooth projective surface with $\rho(S) \geq 3$ and its effective cone is polyhedral cone. Then
$$ \Eff(X) = \sum_{[C] \in \Exc(X)} \mathbb{R}_{\geq 0}[C] $$
where $\Exc(X)$ is the set of classes of integral curves $C$ of $X$ with $C^2 < 0.$
\end{proposition}

Let us define a negative curve on a smooth projective surface as follows.

\begin{definition}
Let $X$ be a smooth projective surface. A negative curve $C$ is an irreducible reduced curve on $X$ such that $C^2 < 0.$
\end{definition}

It is well known that the nef cone is dual to the closure of the effective cone. Therefore if the $\Eff(X)$ is a rational polyhedral cone then the $\Nef(X)$ is also a rational polyhedral cone. In this case, it is sufficient to prove that extremal rays of $\Nef(X)$ is semiample to prove that $X$ is Mori dream space. Let us recall the following result which is helpful to prove a given divisor is semiample.

\subsection{Mori dream space}

Hu and Keel studied relation between minimal model program and variation of GIT and defined the notion of Mori dream space in \cite{HK}. Let us recall the definition of Mori dream space.

\begin{definition}\cite{HK}
A variety $X$ is a Mori dream space if \\
(1) $X$ is a $\mathbb{Q}$-factorial variety and $h^1(X,\mathcal{O}_X)=0,$ \\
(2) the nef cone of $X$ is the convex cone generated by finitely many semiample classes. \\
(3) there are finitely many birational maps $\phi_i : X \dashedrightarrow X_i, 1 \leq i \leq m$ which are isomorphisms in codimension 1, $X_i$ are varieties satisfying (1), (2) and if $D$ is a movable divisor then there is an index $1 \leq i \leq m$ and a semiample divisor $D_i$ on $X_i$ such that $D=\phi^*_iD_i.$
\end{definition}

Let us recall the definition of Cox ring.

\begin{definition}
Let $X$ be a normal projective $\mathbb{Q}$-factorial variety with finitely generated $\Cl(X)$. Let $\Gamma \subset \Cl(X)$ be a free Abelian group such that the inclusion map induces an isomorphism $\Gamma \otimes \mathbb{Q} \cong \Cl(X) \otimes \mathbb{Q}.$ Then a Cox ring of $X$(associated to $\Gamma$) is a multi-graded ring defined as follows.
$$ \Cox(X) = \bigoplus_{D \in \Gamma}H^0(X,\mathcal{O}_X(D)). $$
\end{definition}

\begin{remark}\cite{HK, Okawa}
Note that the definition of a Cox ring depends on the choice of $\Gamma \subset \Cl(X).$ However it is well-known that the finite-generation of a Cox ring of $X$ does not depend on the choice of $\Gamma \subset \Cl(X).$
\end{remark}

It is well-known that a variety $X$ is a Mori dream space if and only if the Cox ring of $X$ is finitely generated (cf. \cite{HK}).

\begin{theorem}\cite{HK}
Let $X$ be a $\mathbb{Q}$-factorial variety such that $\Pic(X)$ is a finitely generated abelian group. Then the followings are equivalent. \\
(1) $X$ is a Mori dream space. \\
(2) $\Cox(X)$ is a finitely generated ring.
\end{theorem}

Let us recall Okawa's theorem which we will use frequently in this paper.

\begin{theorem}\cite{Okawa}
Let $f : X \to Y$ be a surjective morphism between normal $\mathbb{Q}$-factorial projective varieties and $X$ be a Mori dream space. Then $Y$ is also a Mori dream space.
\end{theorem}

\subsection{Mori dream surfaces}

There are simple criterions for a surface to be a Mori dream space. Artebani, Hausen and Laface proved the following theorem in \cite{AHL}.

\begin{theorem}\cite[Theorem 2.5]{AHL}
Let $X$ be a normal complete surface with finitely generated $\Cl(X).$ Then the followings are equivalent. \\
(1) $\Cox(X)$ is finitely generated. \\
(2) The effective cone $\Eff(X) \subset \Cl(X)_{\mathbb{R}}$ and moving cone $\Mov(X) \subset \Cl(X)_{\mathbb{R}}$ are rational polyhedral cones and $\Mov(X)=\SAmp(X).$
\end{theorem}

As a corollary we have the following helpful criterion of finitely generation of Cox rings of $\mathbb{Q}$-factorial surfaces.

\begin{corollary}\cite[Corollary 2.6]{AHL}
Let $X$ be a $\mathbb{Q}$-factorial projective surface with finitely generated $\Cl(X).$ Then the followings are equivalent. \\
(1) $\Cox(X)$ is finitely generated. \\
(2) The effective cone $\Eff(X) \subset \Cl(X)_{\mathbb{R}}$ is a rational polyhedral cone and $\Nef(X)=\SAmp(X).$
\end{corollary}

There are many examples of surfaces with $\kappa \leq 0$ which are Mori dream spaces. For $\kappa = -\infty,$ it is well-known that log del Pezzo surfaces are Mori dream surfaces.

\begin{theorem}\cite[Corollary 1.3.2]{BCHM}
Let $X$ be a log Fano variety. Then $X$ is a Mori dream space. In particular, log del Pezzo surfaces are Mori dream surfaces.
\end{theorem}

For $\kappa = 0,$ the following criterion is well-known.

\begin{theorem}\cite{AHL}
Let $X$ be a K3 surface or an Enriques surface. Then $X$ is a Mori dream surface if and only if $\Aut(X)$ is a finite group.
\end{theorem}

Let us recall following remarks(cf. \cite{Fulton1}).

\begin{remark}
(1) Let $f : X \to Y$ be a finite morphism between two smooth surface, then $K_X=f^*(K_Y)+R$ where $R$ is the ramification divisor. \\
(2) Let $f : X \to Y$ be a finite flat morphism of degree $d.$ Then $$ A_*Y \to A_*X \to A_*Y $$ is multiplication of degree $d.$ \\
(3) Suppose that a finite group $G$ acts on $X$ and the quotient variety is $Y.$ Then there is a canonical isomorphism
$$ (A_*Y)_{\mathbb{R}} \cong (A_*X)^G_{\mathbb{R}} $$
and the natural map
$$ (A_*Y)_{\mathbb{R}} \to (A_*X)^G_{\mathbb{R}} \to (A_*X) _{\mathbb{R}} \to (A_*Y) _{\mathbb{R}} $$
is multiplication of degree $|G|.$
\end{remark}

From the above results we get the following proposition.

\begin{proposition}\label{criterion}
Suppose that $\pi : X \to Y$ is a finite flat morphism of degree $d$ between normal $\mathbb{Q}$-factorial projective surfaces with $h^1(X,\mathcal{O}_X)=0$ and $\pi^* : \Pic(Y) _{\mathbb{R}} \cong \Pic(X) _{\mathbb{R}}$ is an isomorphism whose inverse is $\frac{1}{d} \pi_* : \Pic(X) _{\mathbb{R}} \cong \Pic(Y) _{\mathbb{R}}.$ Then we have the followings. \\
(1) $X$ is a Mori dream surface if and only if $Y$ is also a Mori dream surface. \\
(2) The effective, nef and semiample cones of $X$ are pull-backs of those of $Y.$ \\
(3) When $\Eff(X)$(or $\Eff(Y)$) is a rational polyhedral cone, every negative curve on $X$ is a pullback of a negative curve on $Y.$ Moreover the pull-back of a negative curve on $Y$ does not split.
\end{proposition}
\begin{proof}
Note that the isomorphism $\pi^* : \Pic(Y) _{\mathbb{R}} \cong \Pic(X) _{\mathbb{R}}$ send $\Eff(Y)$(resp. $\Nef(Y)$) into $\Eff(X)$(resp. $\Nef(X)$). Conversely, the isomorphism $\pi_* : \Pic(X) _{\mathbb{R}} \cong \Pic(Y) _{\mathbb{R}}$ send $\Eff(X)$ into $\Eff(Y).$ Therefore we can identify $\Eff(X)$ and $\Eff(Y)$ via the isomorphism $\pi^* : \Pic(Y) _{\mathbb{R}} \cong \Pic(X) _{\mathbb{R}}.$ \\

When $X$ is a Mori dream surface then we see that $Y$ is also a Mori dream surface from Okawa's theorem (cf. \cite{Okawa}). Now suppose that $Y$ is a Mori dream surface. Recall that we have an isomorphism $\pi^* : \Pic(Y) _{\mathbb{R}} \cong \Pic(X) _{\mathbb{R}}.$ Via this isomorphism we can identify the effective cones and nef cones of $X$ and $Y.$ Because $Y$ is a Mori dream surface, we see that the effective cone of $X$ is also a rational polyhedral cone. Let $D$ be a nef divisor on $\Pic(X) _{\mathbb{R}}.$ Because every divisor in $\Pic(X) _{\mathbb{R}}$ is a pull-back of a divisor of $\Pic(Y) _{\mathbb{R}}$ and $\pi$ is surjective, we see that $D$ is a pull-back of a nef divisor $C$ in $\Pic(X) _{\mathbb{R}}$. Because $Y$ is a Mori dream space, we see that $C$ is semiample. Suppose that $\mathbb{B}|D|$ is nonempty. Then $\pi_*\mathbb{B}|D|$ is contained in $\mathbb{B}|C|$ which gives a contradiction. Therefore we see that $X$ is a Mori dream surface. \\

Recall that a negative curve on $X$ lies on an extremal ray of $\Eff(X).$ From the above identification of $\Eff(X)$ and $\Eff(Y)$ via the isomorphism $\pi^* : \Pic(Y) _{\mathbb{R}} \cong \Pic(X) _{\mathbb{R}}$ we see that there is a negative $\mathbb{Q}$-divisor on $Y$ such that its pull-back is the negative curve on $X.$ Let $C'$ be a negative curve on $Y.$ Suppose that $\pi^*(C')$ is not irreducible. Then there is an irreducible component $D_1$ of $\pi^*(C')$ such that $D_1^2<0.$ Let $D_2$ be another irreducible component of $\pi^*(C').$ Because $D_1^2<0, D_1 \cdot D_2 \geq 0$ we see that $D_1, D_2$ are linearly independent vectors in $\Pic(X)_{\mathbb{R}}$ which goes to the same element $C'$ in $\Pic(X)_{\mathbb{R}}.$ This gives a contradiction to the fact that $\pi^* : \Pic(Y) _{\mathbb{R}} \cong \Pic(X) _{\mathbb{R}}.$ Therefore we see that the pull-back of a negative curve on $Y$ does not split.
\end{proof}

\begin{remark}
We found that Okawa obtained more general result than the first part of the above proposition via different method in \cite{Okawa}.
\end{remark}

\section{Minimal surfaces of general type with $p_g=0, 7 \leq K^2 \leq 9.$}

In this section, we discuss Mori dream surfaces of general type with $p_g=0, 7 \leq K^2 \leq 9.$ Let $X$ be a surface of general type with $p_g=0, 7 \leq K^2 \leq 9.$ Because the Picard rank of $X$ is small it is relatively easier to check whether $X$ is a Mori dream space or not.

\subsection{Fake projective planes}

Minimal surfaces of general type with $p_g=0, K^2=9$ are called fake projective planes. Fake projective planes are classifies by works of Prasad and Yeung and Cartright and Steger.

\begin{theorem}\cite{CS,PY, PY_addendum}
There are exactly 100 fake projective planes.
\end{theorem}

From Noether's theorem we see that the Picard rank of any fake projective plane is 1.

\begin{proposition}
Let $X$ be a normal $\mathbb{Q}$-factorial projective variety with $h^1(X,\mathcal{O}_X)=0$ and Picard number 1. Then $X$ is a Mori dream space.
\end{proposition}
\begin{proof}
Because the Picard number of $X$ is 1, we can choose $\Gamma = \langle D \rangle \subset \Cl(X)$ where $D$ is an ample divisor. Therefore the Cox ring of $X$ with respect to $\Gamma$ is isomorphic to the section ring $R(X,D)$ and hence finitely generated.
\end{proof}

Therefore we have the following conclusion.

\begin{corollary}
Every fake projective plane is a Mori dream space.
\end{corollary}

\subsection{Fake quadrics}

Minimal surfaces of general type with $p_g=0, K^2=8$ are called fake quadrics. Unlike the fake projective planes, we do not know how to classify fake quadrics. We also do not know whether all fake quadrics are Mori dream space or not.

\begin{question}
Let $X$ be a fake quadric. When $X$ is a Mori dream space?
\end{question}

Typical examples of fake quadrics are surfaces isogenous to a higher product. Let us recall their definition.

\begin{definition}\cite{Catanese}\label{surface isogenous to a higher product}
A surface $X$ is isogenous to a higher product if $X$ admits a finite unramified covering which is isomorphic to a product of two curves whose genus are greater than or equal to 2.
\end{definition}

Catanese proved that when $X$ is a surface isogenous to a higher product then it belongs to one of two possible types(unmixed type and mixed type) in \cite{Catanese}.
When $X$ is a surface isogenous to a higher product, then there are two curves $C, D$ and finite group $G$ acting on them. The diagonal action of $G$ on $C \times D$ is free and $X$ is isomorphic to $(C \times D)/G.$
These surfaces were classified by Bauer, Catanese and Grunewald in \cite{BCG}. They form an important class of surfaces of general type with $p_g=0, K^2=8.$

\begin{lemma}
Let $X$ be a surface isogenous to a higher product of unmixed type with $p_g=0.$ Then $X$ is a Mori dream space. The effective cone and nef cone are generated by fibers of $X \to C/G$ and fibers of $X \to D/G.$
\end{lemma}
\begin{proof}
Let $X$ be a surface isogenous to a higher product with $p_g=q=0$ of unmixed type. Then we have the following diagram.

\begin{displaymath}
\xymatrix{
 & \ar[ld] C \times D \ar[d]  \ar[rd] &  \\
C \ar[d] & \ar[ld] X \cong (C \times D)/G \ar[rd] & \ar[d] D  \\
C/G \cong \mathbb{P}^1 &  & D/G \cong \mathbb{P}^1 }
\end{displaymath}

Then it is easy to see that $\Nef(X)$ is the convex cone generated by the two line bundles which are pull-back of ample line bundles of the two $\mathbb{P}^1.$ Because these bundles gave fibrations, we see that every nef divisor is semiample. Therefore $X$ is a Mori dream space.
\end{proof}

\begin{question}
(1) Let $S$ be a surface isogenous to a higher product of mixed type. Is $S$ a Mori dream surface? \\
(2) Let $S$ be a irreducible fake quadric. Is $S$ a Mori dream surface?
\end{question}

\subsection{$K^2=7$ cases}

There are few explicitly constructed examples of minimal surfaces of general type with $p_g=0$ and $K^2=7.$ A famous family of such surfaces is the family of Inoue surfaces. Recently, Chen constructed a new family of such surfaces in \cite{Chen1}. We will prove that they are Mori dream surfaces. \\

Let us recall the construction of Inoue surfaces, the first examples of minimal surfaces of general type with $p_g=0, K^2=7.$ Inoue considered product of four elliptic curves with a natural $(\mathbb{Z}/2\mathbb{Z})^5$-actions and smooth invariant complete intersections of divisors of degree $(2,2,2,0),(0,0,2,2).$ Then he constructed Inoue surfaces with $p_g=0, K^2=7$ as free quotients of these complete intersections. Mendes Lopes and Pardini proved that Inoue surfaces can be realized as bidouble coverings over nodal cubic surfaces. Let us recall their construction. We will follow the explanation of \cite{MLP} and see \cite{BC3, Inoue, MLP} for more details. \\

Consider the quadrilateral $p_1p_2p_3p_4$ in $\mathbb{P}^2.$ Let $p_5$ be the intersection of two lines $\overline{p_1p_2}$ and $\overline{p_3p_4}$ and let $p_6$ be the intersection of two lines $\overline{p_1p_4}$ and $\overline{p_2p_3}.$ Let $W \to \mathbb{P}^2$ be the blowup of these six points. Let $\overline{\Delta_1}$ be the strict transform of the line $\overline{p_1p_3},$ $\overline{\Delta_2}$ be the strict transform of the line $\overline{p_2p_4},$ $\overline{\Delta_3}$ be the strict transform of the line $\overline{p_5p_6}.$ They are $(-1)$-curves on $W.$ Let $\overline{c_1}$ be the strict transform of a general conic though $p_2p_4p_5p_6,$ $\overline{c_2}$ be the strict transform of a general conic though $p_1p_3p_5p_6$ and $\overline{c_3}$ be the strict transform of a general conic though $p_1p_2p_3p_4.$ Let $n_i$ be the strict transforms of the line $\overline{p_ip_{i+1}}.$ They are the only nodal curves on $W$ and let $W \to Y$ be the contraction of these four nodal curves.

Let $D_1=\overline{\Delta}_1+\overline{c_2}+n_1+n_2,$ $D_2=\overline{\Delta}_2+\overline{c_3},$ $D_3=\overline{\Delta}_3+\overline{c_1}+\overline{c'_1}+n_3+n_4$ where $\overline{c_1},\overline{c'_1} \in |\overline{c_1}|,\overline{c_2} \in |\overline{c_2}|,\overline{c_3} \in |\overline{c_3}|$ are general elements in the corresponding linear system. Then Mendes Lopes and Pardini showed that $D_1,D_2,D_3$ define a smooth bidouble covering $V$ over $W$ whose branch locus of $D_1+D_2+D_3.$ Then $V$ has eight exceptional curves and contracting these exceptional curves gives the Inoue surface $X.$ Moreover $X$ is a bidouble cover over $Y.$ Therefore we have the following commutative diagram.

\begin{displaymath}
\xymatrix{
V \ar[r] \ar[d] & X \ar[d] \\
W \ar[r] & Y }
\end{displaymath}

\begin{proposition}
Let $X$ be a Inoue surface. Then $X$ is a Mori dream surface.
\end{proposition}
\begin{proof}
Let $X$ be a such Inoue surface and $\pi : X \to Y$ be a bidouble covering. From the construction of \cite{MLP}, $W$ is a weak del Pezzo surface and we see that $Y$ is a surjective image of the Mori dream surface $W.$ Therefore $Y$ is a Mori dream space with $\rho(Y)=3.$ Then by the construction we see that $\pi^* : \Pic(Y) _{\mathbb{R}} \to \Pic(X) _{\mathbb{R}}$ is an isomorphism. Therefore we have the desired results.
\end{proof}

Moreover we can compare canonical bundles of $X$ and $Y$ as follows.

\begin{lemma}\cite{MLP}
We have the following equivalence
$$ 2K_X \sim \pi^*(-K_Y+c_1) $$
where $c_1$ is the image of $\overline{c_1}$ in $Y.$
\end{lemma}

Let us compute effective and nef cones of Inoue surfaces. Because $\rho \geq 3,$ we need some computation in order to determine the shapes of effective and nef cones. We have a simple strategy to compute effective cones and nef cones. Suppose that we know several effective divisors $E_1, \cdots, E_k$ on $X.$ Let $V=\Cl(X)_{\mathbb{R}}$ and $A \subset V$ be the rational polyhedral cone generated by these effective divisors. Suppose that we can check that the extremal rays of $A^\vee \subset V^\vee$ are nef divisors. Because $A \subset \Eff(X)$ we have $\Nef(X) \subset A^\vee \subset \Nef(X)$ and hence we see that $A=\Eff(X)$ and $A^\vee=\Nef(X).$

\begin{proposition}
Let $X$ be a Inoue surface. Then the effective cone of $X$ has three generators which are pullback of three $(-1)$-curves in $Y.$ The pull-back of three $(-1)$-curves are two elliptic curves and one genus 2 curve whose self-intersection numbers are all $-1.$ The nef cone of $X$ has three generators which are pullback of three nef divisors of $Y.$
\end{proposition}
\begin{proof}
To describe nef cone and effective cone of a Inoue surface $X$ is equivalent to describe the same invariants for 4-nodal cubic $Y.$ Because $Y$ is obtained by contracting 4 nodal curves from $Bl_{6 pts} \mathbb{P}^2$ and we know the configuration of curves on this weak del Pezzo surface, we can compute nef cone and effective cone of $Y.$

From \cite{MLP} we see that $Y$ is contraction of six $(-2)$-curves of del Pezzo surface and there are three $(-1)$ curves on $Y.$ Let $\Delta_1$(resp. $\Delta_2,$ $\Delta_3$) be the image of $\overline{\Delta_1}$(resp. $\overline{\Delta_2},$ $\overline{\Delta_3}$). Because these curves are disjoint from the nodal curves in $W,$ they are $(-1)$-curves on $Y.$ We can see that $\Delta_1+\Delta_2, \Delta_2+\Delta_3, \Delta_3+\Delta_1$ are nef divisors on $Y.$ For any triple $\{ i,j,k \} = \{ 1,2,3 \}$ we have the followings
$$ (\Delta_i + \Delta_j) \cdot \Delta_i = 0, $$
$$ (\Delta_i + \Delta_j) \cdot \Delta_j = 0, $$
$$ (\Delta_i + \Delta_j) \cdot \Delta_k > 0. $$

We can see that $\Delta_1+\Delta_2, \Delta_2+\Delta_3, \Delta_3+\Delta_1$ are nef divisors on $Y.$ The rational polyhedral cone generated by $\Delta_1,\Delta_2,\Delta_3$ is a subcone of $\Eff(Y).$ Then the above computation shows that $\Delta_1, \Delta_2, \Delta_3$ generate the effective cone of $Y$ and $\Delta_1+\Delta_2, \Delta_2+\Delta_3, \Delta_3+\Delta_1$ generate the nef cone of $Y.$

Note that $\Delta_1, \Delta_2, \Delta_3$ are branch locus of the map $\pi : X \to Y.$ Therefore $\pi^*\Delta_i = 2\widetilde{\Delta_i}$ for an irreducible divisor $\widetilde{\Delta_i}$ for every $i=1,2,3.$ Because self-intersection of $\pi^*(\Delta_i)$ is $-4,$ we have $\widetilde{\Delta_i}^2=-1$ for $i=1,2,3.$ Moreover we have the following identity
$$ K_X \cdot \widetilde{\Delta_i} = \frac{1}{4} \pi^*(-K_Y+\Delta_2+\Delta_3) \cdot \pi^*{\Delta_i} = (-K_Y+\Delta_2+\Delta_3) \cdot {\Delta_i} $$
 for all $i=1,2,3.$
Because $\widetilde{\Delta_i}$ lies on the 1-dimensional fixed locus of an involution acting on $X$ we see that it is a smooth curve. Therefore they are elliptic curves and a genus $2$ curve whose self-intersections are all $-1.$ They generate $\Eff(X).$ Similarly, pull-backs of $\Delta_1+\Delta_2, \Delta_2+\Delta_3, \Delta_3+\Delta_1$ generate $\Nef(X).$
\end{proof}

Via similar method we can prove that Chen's surfaces are Mori dream surfaces. Let us briefly recall the construction of Chen in \cite{Chen1}. Let $p_0,p_1,p_2,p_3$ be points in $\mathbb{P}^2$ in general position and let $p'_i$ be the infinitely near point over $p_i$ which corresponds to the line $\overline{p_0p_i}$ for $i=1,2,3.$ Let $p$ be a point located outside of lines $\overline{p_0p_i},$ $\overline{p_ip_{i+1}}$ for $i=1,2,3$ and conics $c_1, c_2, c_3$ where $c_i$ is the unique conic passing through $p_i,p_{i+1},p'_{i+1},p_{i+2},p'_{i+2}$ in $\mathbb{P}^2.$ Let $W \to \mathbb{P}^2$ be the blowup of these eight points. Let $E_i$(resp. $E_i'$) be the total transform of $p_i$(resp. $p_i'$) and $H$ be the pull-back of a line in $\mathbb{P}^2.$ Let $\overline{\Gamma}$ be the strict transformation of the line $\overline{p_0p}$ and $\overline{E}$(resp. $E_0$) be the total transforms of $p$(resp. $p_0$). The linear system $|-2K_W-\overline{\Gamma}|$(resp. $|-2K_W-\overline{E}|$) consists of a single $(-1)$-curve, $\overline{B_2}$(resp. $\overline{B_3}$). \\

Consider curves $C_i \sim L-E_0-E_i-E_i'$ and $C_i' \sim E_i-E_i'$ for $i=1,2,3.$ Chen showed that $W$ is a weak del Pezzo surface with degree 1 and the above six curves are only nodal curves on $W.$ Let $Y$ be a surface obtained from $W$ by contracting these six nodal curves. Note that $\overline{E}, \overline{\Gamma}, \overline{B_2}, \overline{B_3}$ are disjoint with the six nodal curves. \\

Then Chen showed that three divisors $F_b+\overline{\Gamma}+C_1+C_1'+C_2+C_2', \overline{B_2}+C_3+C_3', \overline{B_3}$ defines a bidouble covering $\pi : V \to W$ branched over them, where $F_b$ is a smooth fiber of a pencil of lines passing through $p_0.$ There are $(-1)$-curves on $V$ and contracting these $(-1)$-curves we obtain $X$ which is a smooth minimal surface of general type with $p_g=0, K^2=7.$ Then $X$ is a bidouble cover over $Y.$ Let us call such $X$ a Chen's surface. Indeed, it is easy to prove that Chen's surfaces are Mori dream surfaces.

\begin{displaymath}
\xymatrix{
V \ar[r] \ar[d] & X \ar[d] \\
W \ar[r] & Y }
\end{displaymath}

\begin{proposition}
Let $X$ be a Chen's surface. Then $X$ is a Mori dream surface.
\end{proposition}
\begin{proof}
Because $Y$ can be obtain contracting six nodal curves from a weak del Pezzo surfaces of degree 1, we see that $Y$ is a Mori dream surface with $\rho(Y)=3.$ From the construction $X$ has a finite map $X \to Y$ and $\rho(X)=3.$ Therefore $X$ is a Mori dream surface.
\end{proof}

Now let us compute effective cones and nef cones of Chen's surfaces. He computed intersection numbers between $\overline{E}, \overline{\Gamma}, \overline{B_2}, \overline{B_3}$ as follows.


\begin{center}
\begin{tabular}{|c|c|c|c|c|} \hline
$\cdot$ & $\overline{E}$ & $\overline{\Gamma}$ & $\overline{B_2}$ & $\overline{B_3}$ \\
\hline $\overline{E}$ & -1 & 1 & 1 & 3 \\
\hline $\overline{\Gamma}$ & 1 & -1 & 3 & 1 \\
\hline $\overline{B_2}$ & 1 & 3 & -1 & 1 \\
\hline $\overline{B_3}$ & 3 & 1 & 1 & -1 \\
\hline
\end{tabular}
\end{center}

\bigskip

Because they are disjoint from the six nodal curves, there curves are pull-back of curves in $Y$ and the intersection numbers are the same. Let $E$, $\Gamma$, $B_2$, $B_3$ be the image in $Y.$ Of course, the intersections $E,\Gamma,B_2,B_3$ are the same as above. Now we can describe nef cones and effective cones of Chen's surfaces explicitly.

\begin{proposition}
Let $Y$ be a weak del Pezzo surface of degree 1 described as above. Then $\Eff(Y)$ is a rational polyhedral cone generated by $E,\Gamma,B_2,B_3.$ and $\Nef(X)$ is a rational polyhedral cone generated by $\Gamma+B_3,B_2+B_3,E+\Gamma,E+B_2.$
\end{proposition}
\begin{proof}
From \cite{Chen1} we see that $Y$ is contraction of six $(-2)$-curves of the weak del Pezzo surface $W$ and there are four $(-1)$-curves $E, \Gamma, B_2, B_3$ on $Y.$ From the above intersection numbers we see that $\Gamma+B_3,B_2+B_3,E+\Gamma,E+B_2$ are nef divisors. We can directly check that the two cones are dual to each other. Therefore we obtain the desired result from our strategy.
\end{proof}

Now we know the extremal rays of $\Eff(Y).$ Let us discuss their pull-backs. First we can compare the canonical divisors.

\begin{lemma}\cite{Chen1}
We have the following isomorphism.
$$ 2K_X \simeq \pi^*(-2K_Y + \Gamma) $$
\end{lemma}

Now we have the following conclusion.

\begin{proposition}
Let $X$ be a Chen's surface with $K^2=7.$ Then the effective cone of $X$ has four generators which are pullback of four $(-1)$-curves in $Y.$ The pull-back of four $(-1)$-curves are three curves with self-intersection $-1$ whose arithmetic genus are $1,2,3$ and one (arithmetic) genus $2$ curve whose self-intersection number is $-4.$ The nef cone of $X$ has four generators which are pullback of four nef divisors on $Y.$
\end{proposition}
\begin{proof}
Note that $\Gamma, B_2, B_3$ lie on the branch locus of the map $\pi : X \to Y.$ Therefore $\pi^*\Gamma = 2\widetilde{\Gamma}$ for an irreducible divisor $\widetilde{\Gamma}$ and $\pi^*B_i = 2\widetilde{B_i}$ for an irreducible divisor $\widetilde{B_i}$ for $i=1,2.$ Because self-intersections of these curves are all $-4,$ we have $\widetilde{\Gamma}^2=\widetilde{B_2}^2=\widetilde{B_3}^2=-1.$ Moreover we have the following identities.
$$ K_X \cdot \widetilde{\Gamma}= \frac{1}{4} \pi^*(-2K_Y + \Gamma) \cdot \pi^*(\Gamma) = (-2K_Y + \Gamma) \cdot \Gamma = 1 $$
$$ K_X \cdot \widetilde{B_2}= \frac{1}{4} \pi^*(-2K_Y + \Gamma) \cdot \pi^*(B_2) = (-2K_Y + \Gamma) \cdot B_2 = 5 $$
$$ K_X \cdot \widetilde{B_3}= \frac{1}{4} \pi^*(-2K_Y + \Gamma) \cdot \pi^*(B_3) = (-2K_Y + \Gamma) \cdot B_3 = 3 $$

Note that $E$ does not lie on the branch locus of the map $\pi : X \to Y.$ Therefore $\pi^*E = \widetilde{E}$ for an irreducible divisor $\widetilde{E}.$ Because $E^2=-1,$ we have $\widetilde{E}^2=-4.$ Moreover we have the following identity.
$$ K_X \cdot \widetilde{E}= \frac{1}{2} \pi^*(-2K_Y + \Gamma) \cdot \pi^*(E) = 2(-2K_Y + \Gamma) \cdot E = 6 $$
From the proposition \ref{criterion}, we see that $\widetilde{E}, \widetilde{\Gamma}, \widetilde{B_2}, \widetilde{B_3}$ are negative curves generate $\Eff(X).$ Similarly, pull-backs of $\Gamma+B_3,B_2+B_3,E+\Gamma,E+B_2$ generate $\Nef(X).$
\end{proof}

\section{Minimal surfaces of general type with $p_g=0, 2 \leq K^2 \leq 6.$}

When $K_X^2$ becomes smaller, it becomes harder to check whether a given surface $X$ is a Mori dream space or not. However for $2 \leq K^2 \leq 6$ cases there are several classical surfaces of general type with $p_g=0$ for which we can show that they are Mori dream surfaces. Our results in this section were motivated by \cite{Alexeev, AO, Coughlan} and some of the results easily follow from them. Especially, semigroups of effective divisors of some surfaces discussed in this section were computed in \cite{Alexeev, Coughlan}.

\subsection{Abelian coverings of weak del Pezzo surfaces}

Let us recall the definition of the abelian covering.

\begin{definition}
An abelian covering of $Y$ with an abelian group $G$ is a variety $X$ with a faithful action of $G$ on $X$ such that there is a finite morphism $\pi : X \to Y$ which is the quotient map of $X$ by the group action $G.$
\end{definition}

Let $X$ be a surface of general type with $p_g=q=0$ and let $\pi : X \to Y$ be an abelian covering where $Y$ is a del Pezzo surface. There are lots of surfaces of general type with $p_g=q=0$ constructed in this way, e.g. Burniat surfaces, Kulikov surfaces, some numerical Campedelli surfaces, etc. The key property in some of these examples is that there is a natural isomorphism $\Pic(X) _{\mathbb{R}} \cong \Pic(Y) _{\mathbb{R}}$ which preserves the intersection pairing(up to scale) and we can identify effective cones and nef cones of $X$ and $Y.$ This phenomena was observed in \cite{Alexeev, AO, Coughlan} for the some of such surfaces and play key role in their works. In these cases, because $Y$ is a Mori dream space we see that the effective cone of $X$ is a rational polyhedral cone and every nef divisor of $X$ is semiample. Therefore we see that $X$ is a Mori dream space.

\subsection{Weak del Pezzo surfaces}

Let $Y$ be a smooth rational surface. When $-K_Y$ is ample(resp. nef and big) we call $Y$ a (resp. weak) del Pezzo surface. It is well-known that (weak) del Pezzo surfaces are Mori dream surfaces. Let us recall several basic definitions and facts about (weak) del Pezzo surfaces.

\begin{definition}\cite{ADHL}
We say that $1 \leq r \leq 8$ distinct points $p_1,\cdots,p_r$ in $\mathbb{P}^2$ are in general position if they satisfy the following conditions. \\
(1) No three of them lie on a line. \\
(2) No six of them lie on a conic. \\
(3) No eight of them lie on a cubic with a singularity at some of the $p_i.$
\end{definition}

It is well-known that a del Pezzo surface $Y$ is either $\mathbb{P}^1 \times \mathbb{P}^1$ or blow-ups of $\mathbb{P}^2$ at $0 \leq r \leq 8$ points in general position. Let $\phi : Y=Y_r \to Y_{r-1} \to \cdots \to Y_0 = \mathbb{P}^2$ be the blowup at $p_i \in Y_{i-1}.$ Let $E_i$ denotes the total transform of the exceptional divisor over $p_i \in Y_{i-1}$ and let $H$ be the pull-back of $\mathcal{O}_{\mathbb{P}^2}(1).$ We say that a point $p_i$ lie on a line(resp. conic) if its image in $\mathbb{P}^2$ lies on a line(resp. conic).

\begin{definition}\cite{ADHL}
We say that $1 \leq r \leq 8$ (possibly infinitely near) points $p_1,\cdots,p_r$ in $\mathbb{P}^2$ are in almost general position if they satisfy the following conditions. \\
(1) No four of them lie on a line. \\
(2) No seven of them lie on a conic. \\
(3) $E_i$ is a $(-1)$-curve or a chain of rational curves whose last component is a $(-1)$-curve and all the other components are $(-2)$-curves.
\end{definition}

The above condition is equivalent to saying that no $p_i$ lies on the $(-2)$-curve on $Y_{i-1}.$ \\

It is well-known that a weak del Pezzo surface $Y$ is either $\mathbb{P}^1 \times \mathbb{P}^1$ or $\mathbb{F}_2$ or blow-ups of $\mathbb{P}^2$ at $0 \leq r \leq 8$ points in almost general position. Of course, a del Pezzo surface is a weak del Pezzo surface.
Because a weak del Pezzo surface is a Mori dream surface, it is easy to see the following description of its effective cone.

\begin{lemma}
(1) Let $Y$ be a smooth del Pezzo surface with $\rho \geq 3.$ Then $\Eff(Y)$ is the rational polyhedral cone generated by the classes of $(-1)$-curves. \\
(2) Let $Y$ be a smooth weak del Pezzo surface with $\rho \geq 3.$ Then $\Eff(Y)$ is the rational polyhedral cone generated by the classes of $(-1)$-curves and $(-2)$-curves.
\end{lemma}

Therefore in order to descirbe $\Eff(Y)$ explicitly, one need to find all $(-1)$-curves and $(-2)$-curves on $Y$ explicitly. Indeed, these curves were intensively studied by many authors. See \cite{Dolgachev} for more details.

\begin{theorem}\cite[Theorem 8.3.2]{Dolgachev} \label{wdP}
Let $Y$ be a weak del Pezzo surface of degree $d.$ Then we have the followings. \\
(1) If $d \geq 2,$ then $|-K_Y|$ has no base points. \\
(2) Let $\phi$ be the morphism defined by $|-K_Y|.$ If $d \geq 3,$ then the image $\phi(Y)$ is a del Pezzo surface of degree $d$ in $\mathbb{P}^d$ with rational double points. The morphism $\Phi=|-K_Y| : Y \to \mathbb{P}^d$ contracts $(-2)$-curves on $Y.$ \\
(3) If $d=2,$ then $\phi$ factors $Y \to \overline{Y} \to \mathbb{P}^2$ where $\overline{Y}$ is a normal surface and $\overline{Y} \to \mathbb{P}^2$ is a finite map of degree 2 branched along a curve $B.$ The image of a component of the chains of $(-2)$-curves on $Y$ is a rational double point on $\overline{Y}.$ The curve $B$ is either smooth or has only simple singularities.
\end{theorem}

We have the following characterization of $(-2)$-curves on weak del Pezzo surfaces.

\begin{lemma}\cite{ADHL}
Let $Y$ be a weak del Pezzo surface which is the blow-up of $\mathbb{P}^2$ at $2 \leq r \leq 8$ points in almost general position. Then any $(-2)$-curve on $Y$ is either of the form $E_i-E_{i+1}$(if $E_i$ is reducible), or is linearly equivalent to one of the $H-E_1-E_2-E_3,2H-E_1-E_2-E_3-E_4-E_5-E_6,3H-2E_1-E_2-E_3-E_4-E_5-E_6-E_7-E_8$(up to permutation of the indices).
\end{lemma}

From the above characterization of $(-2)$-curves, we can explicitly describe negative curves on weak del Pezzo surfaces.

\subsection{Burniat surfaces}

Burniat surfaces can be constructed as bidouble coverings of del Pezzo surfaces with the same Picard rank. Let $p_1,p_2,p_3$ are points in general position in $\mathbb{P}^2.$ For each $p_i,$ there are two lines $\overline{p_{i-1}p_i}, \overline{p_ip_{i+1}}$ (indices modulo 3) and let us consider two distinct lines different from $\overline{p_{i-1}p_i}, \overline{p_ip_{i+1}}.$ Then we have a configuration of nine lines on $\mathbb{P}^2.$ By blowing up points where more than two lines are passing through, we obtain a weak del Pezzo surface $Y.$ From the configuration of lines, one can see that there is a smooth bidouble covering $X$ of $Y.$ These surfaces are called Burniat surfaces and see \cite{Alexeev, BC1, Coughlan} for more details about them. Let $X$ be a Burniat surface and $\pi : X \to Y$ be the quotient map of the bidouble covering. In this case, the pull-back $\pi^* : \Pic(X) _{\mathbb{R}} \to \Pic(Y) _{\mathbb{R}}$ is an isomorphism and we see that $X$ is a Mori dream space from Proposition \ref{criterion}. We can compute the effective cone of $X$ via that of $Y.$ See \cite{Alexeev, AO, Coughlan} for more details and we will follow notations in \cite{BC1}.

\begin{lemma}\cite{Alexeev}
We have the following isomorphism.
$$ 2K_X \simeq \pi^*(-K_Y)$$
\end{lemma}

Let $D$ be a irreducible reduced curve on $Y$ such that $\pi^*D$ does not split. Let $\widetilde{D}$ be the reduced component of $\pi^*D.$ When $D$ is a component of the branch locus of $\pi$ then we have $\pi^*D=2\widetilde{D}$ and we can compute intersection number of $K_X$ and $\widetilde{D}$ as follows.
$$ K_X \cdot \widetilde{D} = \frac{1}{4} \cdot \pi^*(-K_Y) \cdot \pi^*D = (-K_Y) \cdot D$$
$$ \widetilde{D}^2 = \frac{1}{4} \cdot \pi^*D \cdot \pi^*D = D^2 $$

When $D$ is not a component of the branch locus of $\pi$ then we have $\pi^*D=\widetilde{D}$ and we can compute intersection number of $K_X$ and $\widetilde{D}$ as follows.
$$ K_X \cdot \widetilde{D} = \frac{1}{2} \cdot \pi^*(-K_Y) \cdot \pi^*D = 2 \cdot (-K_Y) \cdot D$$
$$ \widetilde{D}^2 = (\pi^*D)^2 = 4D^2 $$

\subsubsection{Burniat surfaces with $K^2=6$}

Burniat surfaces with $K^2=6$ are called primary Burniat surfaces. Let $p_1,p_2,p_3 \in \mathbb{P}^2$ be three points in general position. Then $Y$ is a blowup of these three points on $\mathbb{P}^2.$ Then $Y$ has three exceptional curves and three strict transformations of $\overline{p_ip_j}$ where $i, j \in \{ 1,2,3 \}$ and $i \neq j.$ These six curves are only $(-1)$-curves on $Y.$

\begin{proposition}
The effective cone of $Y$ is the rational polyhedral cone generated by these six $(-1)$-curves.
\end{proposition}

We have the following conclusion.

\begin{proposition}
Let $X$ be a Burniat surface with $K^2=6$ constructed as a bidouble covering over $Y.$ Then $X$ is a Mori dream surface and the effective cone is the rational polyhedral cone which is the pull-back of the effective cone of $Y.$ The six negative curves are elliptic curves whose self-intersection numbers are all $-1.$
\end{proposition}

\subsubsection{Burniat surfaces with $K^2=5$}

Burniat surfaces with $K^2=5$ are called secondary Burniat surfaces. Let $p_1,p_2,p_3,p_4 \in \mathbb{P}^2$ be four points in general position. Then $Y$ is a blow-up of these four points on $\mathbb{P}^2.$ Then $Y$ has four exceptional curves and six strict transformations of $\overline{p_ip_j}$ where $i, j \in \{ 1,2,3,4 \}$ and $i \neq j.$ These ten curves are only $(-1)$-curves on $Y.$

\begin{proposition}
The effective cone of $Y$ is the polyhedral cone generated by the above ten $(-1)$-curves.
\end{proposition}

Then we have the following conclusion.

\begin{proposition}
Let $X$ be a Burniat surface with $K^2=5$ constructed as a bidouble covering over $Y.$ Then $X$ is a Mori dream surface and the effective cone is the rational polyhedral cone which is the pull-back of the effective cone of $Y.$ The ten negative curves are nine elliptic curves with self-intersection numbers are all $-1$ and one negative curve with self-intersection $-4$ and arithmetic genus $0.$
\end{proposition}
\begin{proof}
The ten $(-1)$-curves on $Y$ lie on the branch locus of $\pi : X \to Y$ except exceptional divisor over $p_4.$ The pull-back of the exceptional divisor over $p_4$ is a negative curve with self-intersection $-4$ and arithmetic genus $0.$ The pull-back of other nine $(-1)$-curves are smooth elliptic curve with self-intersection number $-1$ since they lie on the branch locus of $\pi.$
\end{proof}

\subsubsection{Burniat surfaces with $K^2=4$}

Bauer and Catanese proved that there are two types of families of Burniat surfaces with $K^2=4$ (nodal and non-nodal types). See \cite{BC1} for more details. \\


Let us first consider non-nodal cases. From \cite{BC1} we see that $Y$ is a del Pezzo surface. Let $p_1,p_2,p_3,p_4,p_5 \in \mathbb{P}^2$ be five points in general position. Then $Y$ is a blowup of these five points on $\mathbb{P}^2.$ Then $Y$ has five exceptional divisors and ten strict transformations of $\overline{p_ip_j}$ where $i, j \in \{ 1,2,3,4,5 \}$ and $i \neq j.$ There is a unique conic passing through all $p_1,p_2,p_3,p_4,p_5$ and its strict transform gives one more $(-1)$-curve on $Y.$ These sixteen curves are only $(-1)$-curves on $Y.$ Therefore we have the following.

\begin{proposition}
The effective cone of $Y$ is the polyhedral cone generated by sixteen $(-1)$-curves.
\end{proposition}

Then we have the following conclusion.

\begin{proposition}
Let $X$ be a Burniat surface with $K^2=4$ of non-nodal type constructed as above. Then $X$ is a Mori dream surface and the effective cone is the rational polyhedral cone which is the pull-back of the effective cone of $Y.$ There are twelve elliptic curves whose self-intersections are $-1$ and four curves with self-intersections $-4$ and arithmetic genus $0.$
\end{proposition}
\begin{proof}
The strict transforms of nine lines and three exceptional divisors over $p_1, p_2, p_3$ are components of the branch locus of $\pi.$ Therefore their reduced pull-backs are elliptic curves with self-intersection $-1.$ The strict transform of the line $\overline{p_4p_5}$ and exceptional divisors over $p_4, p_5$ are not components of the branch locus of $\pi.$ Therefore their reduced pull-backs are curves with self-intersection $-4$ and arithmetic genus $0.$ Similarly, the strict transform of the the unique conic passing through $\{ p_1,p_2,p_3,p_4,p_5 \}$ is not a component of the branch locus of $\pi.$ Therefore its reduced pull-back is a curve with self-intersection is $-4$ and arithmetic genus $0.$
\end{proof}

Now let us consider the nodal case. From \cite{BC1, Dolgachev} we see that $Y$ is a weak del Pezzo surface whose anticanonical model has a unique $A_1$ singularity. Therefore there is a unique $(-2)$-curve on $Y$ which is the strict transform of the line passing through $p_1,p_4,p_5.$ Moreover we can see that there is no conic passing through all $p_1,p_2,p_3,p_4,p_5$ as follows.

\begin{lemma}
There is no curve in $|2H-E_1-E_2-E_3-E_4-E_5|.$
\end{lemma}
\begin{proof}
Suppose that there is a such curve and let $C$ be the image of the curve in $\mathbb{P}^2.$ From the configuration we see that there is a line passing through three points among these five points. Then $C$ and the line meet at three points and this contradicts to Bezout's theorem.
\end{proof}

There are seven $(-1)$-curves which are strict transforms of lines passing through only two points of $\{ p_1,p_2,p_3,p_4,p_5 \}.$ There are five exceptional divisors on $Y.$ Therefore we can describe the effective cone of $Y$ as follows.

\begin{proposition}
The effective cone of $Y$ is the polyhedral cone generated by above thirteen $(-1)$-curves and the unique $(-2)$-curve.
\end{proposition}

Then we have the following conclusion.

\begin{proposition}
Let $X$ be a Burniat surface with $K^2=4$ of nodal type constructed as above. Then $X$ is a Mori dream surface and the effective cone is the pull-back of the effective cones of $Y.$ There are ten elliptic curves whose self-intersections are $-1,$ a curve with self-intersection is $-2$ and arithmetic genus $0$ and two curves with self-intersection is $-4$ and arithmetic genus $0.$
\end{proposition}
\begin{proof}
The exceptional divisors over $p_1, p_2, p_3$ and the strict transforms of seven lines passing through two points among $\{p_1,\cdots,p_5\}$ are components of the branch locus of $\pi.$ Therefore their reduced pull-backs are elliptic curves with self-intersection $-1.$ The strict transform of the line $\overline{p_1p_4p_5}$ is also a component of the branch locus of $\pi.$ Therefore its reduced pull-back is a curve with self-intersection is $-2$ and arithmetic genus $0.$ The exceptional divisors over $p_4, p_5$ are not branched locus of $\pi$ and their reduced pull-backs are curves with self-intersection $-4$ and arithmetic genus $0.$ Therefore we have the desired result.
\end{proof}

\subsubsection{Burniat surfaces with $K^2=3$}

From Theorem \ref{wdP}, we see that $Y$ has a morphism $|-K_Y| : Y \to \mathbb{P}^3$ where the image $\overline{Y}$ is a cubic surface on $\mathbb{P}^3.$ It is known that $\overline{Y}$ has $3A_1$ singularities. Therefore $Y$ has three $(-2)$-curves which are strict transforms of lines passing through three ponts among $p_1,p_2,p_3,p_4,p_5,p_6.$ Again we can see that there is no conic passing through five points among $p_1,p_2,p_3,p_4,p_5,p_6$ as follows.

\begin{lemma}
There is no curve in $|2H-E_1-E_2-E_3-E_4-E_5|.$
\end{lemma}
\begin{proof}
Suppose that there is a such curve and let $C$ be the image of the curve in $\mathbb{P}^2.$ For any set of five points in $\{p_1,\cdots,p_6\},$ there is a line passing through three points among these five points. Then $C$ and the line meet at three points and this contradicts to Bezout's theorem.
\end{proof}

There are six exceptional divisors and six $(-1)$ curves which are strict transforms of lines passing through only two points among $\{ p_1,p_2,p_3,p_4,p_5 \}.$ Therefore we can describe the effective cone of $Y$ as follows.

\begin{proposition}
The effective cone of $Y$ is the polyhedral cone generated by the above twelve $(-1)$-cuves and three $(-2)$-curves.
\end{proposition}

Then we have the following conclusion.

\begin{proposition}
Let $X$ be a Burniat surface with $K^2=3$ constructed as a bidouble covering over $Y.$ Then $X$ is a Mori dream surface and the effective cone is the rational polyhedral cone which is the pull-back of the effective cone of $Y.$ There are nine elliptic curves whose self-intersections are $-1,$ three curves with self-intersection is $-2$ and arithmetic genus $0$ and three curves with self-intersection is $-4$ and arithmetic genus $0.$
\end{proposition}
\begin{proof}
The exceptional divisors over $p_1, p_2, p_3$ and the strict transforms of six lines passing through two points among $\{p_1,\cdots,p_6\}$ are components of the branch locus of $\pi.$ Therefore their reduced pull-backs are elliptic curves with self-intersection $-1.$ The strict transform of the line $\overline{p_1p_4p_5}, \overline{p_2p_4p_6}, \overline{p_3p_5p_6}$ are also components of the branch locus of $\pi.$ Therefore its reduced pull-backs are curves with self-intersection is $-2$ and arithmetic genus $0.$ The exceptional divisors over $p_4, p_5, p_6$ are not branched locus of $\pi$ and their reduced pull-backs are curves with self-intersection $-4$ and arithmetic genus $0.$ Therefore we have the desired result.
\end{proof}

\subsubsection{Burniat surfaces with $K^2=2$}

Because $Y$ is a weak del Pezzo surface of degree 2. From Theorem \ref{wdP}, we see that $Y$ has a two-to-one map $Y \to \mathbb{P}^2$ which factors $Y \to \overline{Y} \to \mathbb{P}^2.$
From \cite{BC1} we see that the branch locus is union of four lines in general position.
It is known that $Y$ has only nodal singularities. From \cite{BC1, Dolgachev} we see that $\overline{Y}$ has $6A_1$ singularities and $Y$ has only six $(-2)$-curves. From the configuration, $Y$ has seven exceptional curves and strict transforms of $\overline{p_1p_2}, \overline{p_1p_3}, \overline{p_2p_3}.$ We can check that these sixteen curves are only negative curves on $Y$ from the followings.

\begin{lemma}
There is no curve in $|2H-E_1-E_2-E_3-E_4-E_5|.$
\end{lemma}
\begin{proof}
Suppose that there is a such curve and let $C$ be the image of the curve in $\mathbb{P}^2.$ For five points in $\{p_1,\cdots,p_7\},$ there is a line passing through three points among these five points. Then $C$ and the line meet at three points and this contradicts to Bezout's theorem.
\end{proof}

\begin{lemma}
There is no curve in $|3H-2E_1-E_2-E_3-E_4-E_5-E_6-E_7|.$
\end{lemma}
\begin{proof}
Suppose that there is a such curve. Because there is a line passing though all of $p_1,p_6,p_7,$ we see that there is a curve in $|H-E_1-E_6-E_7|.$ Because $(3H-2E_1-E_2-E_3-E_4-E_5-E_6-E_7)(H-E_1-E_2-E_3)=-1$ and the element in $|H-E_1-E_6-E_7|$ is irreducible, we see that there is an element in $|2H-E_1-E_2-E_3-E_4-E_5|.$ However this conclusion contradicts to the previous lemma.
\end{proof}

Using similar argument, we see that there are only three $(-1)$-curves on $Y$ which are strict transforms of $\overline{p_1p_2}, \overline{p_1p_3}, \overline{p_2p_3}.$

\begin{proposition}
The effective cone of $Y$ is the polyhedral cone generated by the above sixteen negative curves.
\end{proposition}

Therefore we have the following conclusion.

\begin{proposition}
Let $X$ be a Burniat surface with $K^2=2$ constructed as a bidouble covering over $Y.$ Then $X$ is a Mori dream surface and the effective cone is the rational polyhedral cone which is the pull-back of the effective cone of $Y.$ There are six elliptic curves whose self-intersections are $-1,$ four curves with self-intersection is $-4$ and arithmetic genus $0$ and six curves with self-intersection $-2$ and arithmetic genus $0.$
\end{proposition}
\begin{proof}
The strict transforms of three $(-1)$-curves $\overline{p_1p_2}, \overline{p_1p_3}, \overline{p_2p_3}$ are components of the branch locus of $\pi.$ Therefore their reduced pull-backs are elliptic curves with self-intersection $-1.$ Among the seven exceptional divisors on $Y,$ three of them lie on the branch locus of $\pi$ and hence their reduced pull-backs are elliptic curves whose self-intersection numbers are all $-1.$ Four of the exceptional divisors do not lie on the branch locus of $\pi$ and hence their reduced pull-backs are curves with self-intersection number $-4$ and arithmetic genus $0$ curves. There are six $(-2)$-curves on $Y$ which lie on the branch locus and hence their reduced pull-backs are curves with self-intersection $-2$ and arithmetic genus $0.$ Therefore we obtain the conclusion.
\end{proof}

\subsection{Kulikov surfaces}

Kulikov surfaces are $(\mathbb{Z}/3\mathbb{Z})^2$-covering of del Pezzo surfaces with degree 6. Because both surfaces have Picard rank 4, we see that our criterion works for these surfaces. Let $X$ be a Kulikov surface and $\pi : X \to Y$ be the quotient map of the $(\mathbb{Z}/3\mathbb{Z})^2$-covering. In this case, the pull-back $\pi^* : \Pic(X) _{\mathbb{R}} \to \Pic(Y) _{\mathbb{R}}$ is an isomorphism and we see that $X$ is a Mori dream space. See \cite{Coughlan} for more details.

\begin{lemma}\cite{Coughlan}
We have the following numerical equivalence.
$$ 3K_X \sim_{num} \pi^*(-K_Y) $$
\end{lemma}

Therefore we obtain the following conclusion.

\begin{proposition}
Let $X$ be a Kulikov surface constructed as above. Then $X$ is a Mori dream surface and the effective cone and nef cone of $X$ are pull-back of those of $Y.$ The six negative curves on $X$ are elliptic curves with self-intersection $-1.$
\end{proposition}
\begin{proof}
From Proposition \ref{criterion} and construction, it is straightforward that $X$ is a Mori dream surface since $Y$ is a del Pezzo surface of degree 6. Therefore $\Eff(Y)$ is generated by six $(-1)$-curves where three of them are exceptional divisors $E_1,E_2,E_3$ and three of them $L_{1},L_{2},L_{3}$ are strict transforms of lines passing through two points among the three blow-up centers. Therefore $\Eff(X)$ is a rational polyhedral cone generated by pullbacks of the six $(-1)$-curves on $Y.$

Now let us compute intersection numbers. Let $\widetilde{E_i}$ be the reduced pull-back of $E_i$ and $\widetilde{L_{i}}$ be the reduced pull-back of $L_{i}$ for $i=1,2,3.$ We have $K_Y \cdot E_i=-1$ and $K_Y \cdot L_i=-1$ since they are $(-1)$-curves. They are belong to the branch locus of $\pi.$ Therefore we have the following identities
$$ K_X \cdot \widetilde{E_i} = \frac{1}{9} \pi^*(-K_Y) \cdot \pi^*{E_i} = (-K_Y) \cdot {E_i} = 1 $$
$$ K_X \cdot \widetilde{L_i} = \frac{1}{9} \pi^*(-K_Y) \cdot \pi^*{L_i} = (-K_Y) \cdot {L_i} = 1 $$
$$ \widetilde{E_i} \cdot \widetilde{E_i} = \frac{1}{9} \pi^*{E_i} \cdot \pi^*{E_i} = {E_i}^2 = -1 $$
$$ \widetilde{L_i} \cdot \widetilde{L_i} = \frac{1}{9} \pi^*{L_i} \cdot \pi^*{L_i} = {L_i}^2 = -1 $$
for all $i \in \{1,2,3 \}.$ Therefore we obtain the desired result.
\end{proof}

\section{Product-quotient surfaces}

Product-quotient surfaces form an important classes of algebraic surfaces and provide many examples of surfaces of general type with $p_g=0.$ Bauer, Catanese, Grunewald and Pignatelli classified minimal product-quotient surfaces with $p_g=0$ in \cite{BCGP, BP}. In this section, we study effective, nef and semiample cones of some product-quotient surfaces with $p_g=0.$ From this we prove that several product-quotient surfaces with $p_g=0$ are Mori dream surfaces.

\subsection{general properties}
Let us recall basic definition and results about product-quotient surfaces.

\begin{definition}\cite{BCGP}\label{product-quotient surface}
Let $G$ be a finite group and $C$, $D$ be algebraic curves with faithful $G$-action. Consider the diagonal action of $G$ on $C \times D.$ An algebraic surface $X$ which is isomorphic to the minimal resolution of $(C \times D)/G$ is called a product-quotient surface and $(C \times D)/G$ is called the quotient model of $X.$
\end{definition}

\begin{displaymath}
\xymatrix{
X \ar[rd] & & \ar[ld] C \times D  \\
 & (C \times D)/G & }
\end{displaymath}

As we proved, product-quotient surfaces with $p_g=q=0, K^2=8$ are Mori dream spaces and we are going to find more product-quotient surfaces which are Mori dream spaces. Product-quotient surfaces with $p_g=q=0$ can be studied via many ways. Let $X$ be a product-quotient surface with $p_g=q=0,$ i.e. minimal resolution of $(C \times D)/G.$ Then we have the following diagram.

\begin{displaymath}
\xymatrix{
X \ar[rd] & \ar[ld] C \times D \ar[d]  \ar[rd] &  \\
C \ar[d] & \ar[ld] (C \times D)/G \ar[rd] & \ar[d] D  \\
C/G \cong \mathbb{P}^1 &  & D/G \cong \mathbb{P}^1 }
\end{displaymath}

Therefore there are two natural fibration maps to projective lines and many geometric information can be extracted from group action of the product of curves. Sometimes, we can compute effective, nef and semiample cones of product-quotient surfaces via these fibration structures. In particular, we can compute the ($\mathbb{Q}$-)basis of Picard groups of the product-quotient surfaces from the two fibration structures. When the $G$-action on $C \times D$ is free(or equivalently $K^2=8$), it is easy to see that the Picard lattice is the hyperbolic plane $H.$ Therefore from now on we study $K^2 \leq 6$ cases. The fibers of two fibrations to projective lines were studied by Serrano in \cite{Serrano}. Let us recall results in \cite{Serrano}.

\begin{theorem}\cite[Theorem 2.1]{Serrano}
Let $c$ in a point of $C$ and $\bar{c}$ be its image of $C \to C/G.$ \\
(1) The reduced fiber of $X \to C/G$ over $\bar{c}$ is the union of an irreducible smooth curve $F_1$, called the central component, and either none or at least two mutually disjoint Hirzebruch-Jung strings, each one meeting the central component at one point. These stings are in one-to-one correspondence with the branch points of $D \to D/G_c$ where $G_c$ is the stabilizer group of $c \in C.$ \\
(2) The intersection of a Hirzebruch-Jung string with $F$ is transversal and takes place at only one of the end components of the string. \\
(3) $F_1$ is isomorphic to $D/G_c$ and has multiplicity equal to $|G_c|$ in the fiber. \\
(4) Let $E=E_1+\cdots+E_k$ be an Hirzebruch-Jung string ordered linearly in the fiber over $\bar{c}$ and consider its image $\bar{d}$ of the another fibration $X \to D/G.$ Let $G_1$ be the central component of the fiber of $X \to D/G$ over $\bar{d}.$ Then $E$ meets $F_1$ and $G_1$ at opposite ends.
\end{theorem}

The self-intersection of strict transform of the reduced fiber was computed by Polizzi in \cite{Polizzi2}. Let us recall a lemma which is very useful to compute the Picard lattice.

\begin{proposition}\cite[Proposition 2.8]{Polizzi2} \cite[Lemma 5.3]{BCGP}
Let $F$ be a reduced fiber of $(C \times D)/G \to C/G$ as a Weil divisor and let $\widetilde{F}$ be its strict transform in $X.$ Suppose that each singular point $x_i \in (C \times D)/G$ on $F$ is of type $\frac{1}{n_i}(1,k_i).$ Then we have the following identity.
$$ -\widetilde{F}^2=\sum_{x_i \in F} \frac{k_i}{n_i}. $$
\end{proposition}

When we compute the effective and nef cone of a product-quotient surface, our next task is to compare its semiample cone. Usually it is a very hard task. To prove some divisors are semiample, we construct explicit automorphism of some product-quotient surfaces. It seems that these automorphisms will have further applications. Let us consider the following situation.

Let $C, D, G$ as above and suppose that the center of $G$ is a nontrivial group containing a nontrivial subgroup $Z.$ From the following isomorphism
$$ G \times Z \cong \{ (g,gz) \in G \times G ~ | ~ g \in G, z \in Z \} \leq G \times G $$
we can see $G \times Z$ as a subgroup of $G \times G$ containing $\Delta{G}.$ Then there is a natrual action of $G \times Z$ on $C \times D$ as follows.
$$ (G \times Z) \times (C \times D) \to (C \times D) $$
$$ (g,gz) \cdot (c,d)=(gc,gzd) $$
where $g \in G, z \in Z, c \in C, d \in D.$ We can also consider a natrual action of $G \times Z$ on $C \times D$ as follows
$$ (G \times Z) \times C \to C $$
$$ (g,gz) \cdot c = gc $$
and a natrual action of $G \times Z$ on $D$ as follows.
$$ (G \times Z) \times D \to D $$
$$ (g,gz) \cdot d = gzd. $$
It is easy to check that the projection maps $(C \times D) \to C$ and $(C \times D) \to D$ are $G \times Z$-equivariant and we have the following commutative diagram.

\begin{displaymath}
\xymatrix{
X \ar[rd] & \ar[ld] C \times D \ar[d]  \ar[rd] &  \\
C \ar[d] & \ar[ld] (C \times D)/G \ar[rd] \ar[d] & \ar[d] D  \\
C/G \cong \mathbb{P}^1 \ar[d] & (C \times D)/(G \times Z)  \ar[ld] \ar[rd]  & D/G \cong \mathbb{P}^1 \ar[d] \\
C/(G \times Z) \cong \mathbb{P}^1 & & D/(G \times Z) \cong \mathbb{P}^1}
\end{displaymath}

Note that the $Z$-actions on $C/G$ and $D/G$ are trivial and hence the $Z$-action on $(C \times D)/G$ preserve fibers of $(C \times D)/G \to C/G$ and $(C \times D)/G \to D/G.$

Suppose that the $G \times Z$-action on $C \times D$ induces a $Z$-action on $X.$ Let $L$ be a $Z$-invariant divisor on $X.$ Then some multiples of $L$ are the pullback of some divisors on $X/Z$ and suppose that one of such divisor is semiample on $X/Z.$ Then we see that $L$ is also semiample. Using this method, we can check some nef divisors become semiample. \\

\subsection{Product-quotient surfaces : $K^2=6,$ $G=D_4 \times \mathbb{Z}/2\mathbb{Z}$ case}

Let $X$ be a product-quotient surface with $p_g=q=0$ and $K^2=6.$ An explicit description of such surfaces was provided in \cite{BCGP}. Let us recall the following diagram.

\begin{displaymath}
\xymatrix{
X \ar[rd] & & \ar[ld] C \times D  \\
& (C \times D)/G & }
\end{displaymath}

In this case, $C$ is a curve with genus 3 and $D$ is a curve with genus 7 with $G$-action. From \cite{BCGP} we see that there are two singular points of type $2 \times \frac{1}{2}(1,1)$ on $(C \times D)/G$ and the $G$-action is encoded in the following data $t_1 : (2,2,2,2,4),$ $S_1 : (56), (56), (12)(34)(56),$ $(13)(56), (1432)$ and $t_2 : (2,2,2,4),$ $S_2 : (24), (14)(23),$ $(13)(24)(56), (1432)(56)$ where $G=\langle (1234), (12)(34), (56) \rangle \leq S_6.$ \\

We can describe the effective cones and nef cones of $X$ explicitly. Let $E_1,E_2$ be the 2 exceptional divisors in $X$ and let $F_1$ be the reduced fiber of one fibration to $\mathbb{P}^1$ meeting $E_1,E_2$ and let $G_1$ be the reduced fiber of another fibration to $\mathbb{P}^1$ meeting $E_1,E_2.$ Then one can check that $F_1$ is the reduced fiber corresponds to the element $(1432).$ Let $F_2$(resp, $F_3,F_4,F_5$) be the reduced fiber corresponds to the element $(56)$(resp. $(56),(12)(34)(56),(13)(56)$). Similarly one can check that $G_1$ is the reduced fiber corresponds to the element $(1432)(56).$ Let $G_2$(resp, $G_3,G_4$) be the reduced fiber corresponds to the element $(24)$(resp. $(14)(23), (13)(24)(56)$). It is easy to see that $E_1,E_2,F_1,G_1$ form a basis of $\Pic(X) _{\mathbb{R}}.$ We can compute intersections between these curves from the results of Serrano and Polizzi as follows. \\

\begin{center}
\begin{tabular}{|c|c|c|c|c|} \hline
$\cdot$ & $E_1$ & $E_2$ & $F_1$ & $G_1$ \\
\hline $E_1$ & -2 & 0 & 1 & 1 \\
\hline $E_2$ & 0 & -2 & 1 & 1\\
\hline $F_1$ & 1 & 1 & -1 & 0 \\
\hline $G_1$ & 1 & 1 & 0 & -1 \\
\hline
\end{tabular}
\end{center}

From the above intersection matrix we have the following isomorphism.

\begin{lemma}
We have the following isomorphism.
$$ K_X \sim_{num} 2E_1+2E_2+3F_1+G_1 $$
\end{lemma}
\begin{proof}
It follows from adjunction formula.
\end{proof}

Moreover we can compute the effective cone of $X.$

\begin{lemma}
The effective cone of $X$ is a rational polyhedral cone generated by $E_1,E_2,F_1,G_1.$
\end{lemma}
\begin{proof}
One can check that $F_1+E_1+G_1,$ $F_1+E_2+G_1,$ $E_1+E_2+2F_1,$ $E_1+E_2+2G_1$ are nef divisors because they are effective divisors whose intersection with any of their component is nonnegative. Let $e_1E_1+e_2E_2+f_1F_1+g_1G_1$ be an element in $\Eff(X).$ Intersecting this divisor with the above nef divisors, one can check that $e_1,e_2,f_1,g_1 \geq 0.$ Therefore we see that $\Eff(X)$ is a rational polyhedral cone generated by $E_1,E_2,F_1,G_1.$
\end{proof}

From the previous lemma, we can also compute $\Nef(X).$

\begin{lemma}
The nef cone of $X$ is a rational polyhedral cone generated by $F_1+E_1+G,$ $F+E_2+G_1,$ $E_1+E_2+2F_1,$ $E_1+E_2+2G_1.$
\end{lemma}
\begin{proof}
We know that the effective cone of $X$ is a rational polyhedral cone generated by $E_1,E_2,F_1,G_1.$ Because the nef cone is the dual polyhedral cone of it, we get the desired result by direct computation.
\end{proof}

To prove that the nef divisors are semiample let us consider involutions on $X.$ It is obvious that $Z=\langle (13)(24), (56) \rangle \leq G=\langle (1234), (12)(34), (56) \rangle \leq S_6$ is a center of $G.$ Therefore we see that there are commuting involutions on $X.$ From group theoretic data associated to these involutions we have the following.

\begin{lemma}
The Picard number of $X/Z$ is 4.
\end{lemma}
\begin{proof}
Note that $F_1$ is isomorphic to $C/\langle (1432) \rangle$ and $G_1$ is isomorphic to $D/\langle (1432)(56) \rangle.$ Then from group theoretic data and the construction of the $\langle (56) \rangle$-action we see that the involution does not change the exceptional locus. Similarly, $\langle (13)(24) \rangle$-action does not change the exceptional locus. Because both actions preserve fibration structure, we see that the Picard number of $X/Z$ is 4.
\end{proof}

We can also compute the 1-dimensional ramification locus of the $Z$-action on $X.$

\begin{lemma}
The 1-dimensional ramification locus of $Z$-action is $F_1+F_2+F_3+G_1+G_4+E_1+E_2.$
\end{lemma}
\begin{proof}
We can directly compute the fixed locus of each action from the group theoretic data and check that $F_1+F_2+F_3+G_1+G_4$ belong to the ramification locus. The only nontrivial part is to check that $E_1$ and $E_2$ belong to the ramification locus. Suppose that $E_1$ does not belong to the ramification locus. Then $E_1=\pi^*E_1''$ for an integral divisor $E''$ of $X/Z$ with $E_1''^2=-\frac{1}{2}.$ If the $\langle (56) \rangle$-action interchange the two intersection points $E_1 \cdot F_1$ and $E_1 \cdot G_1$ then the image of ramification locus of $\langle (13)(24) \rangle$-action in $X/Z$ is not smooth. Therefore the fixed points of the $Z$-action on $E_1$ should be two intersection points with $E_1 \cdot F_1$ and $E_1 \cdot G_1$ (cf. \cite[Remark 2.1]{Chen2}). Because the 1-dimension branch locus of $X/Z$ lie on the smooth locus, we see that $E_1''$ lies on the smooth locus on $X/Z.$ On the other hand $E_1''^2=-\frac{1}{2}$ and we have a contradiction. Therefore $E_1$ should belong to the ramification locus. Similarly we can see that $E_2$ also belongs to the ramification locus.
\end{proof}

Therefore we obtain the desired result.

\begin{theorem}
Let $X$ be a product-quotient surface with $p_g=q=0$ and $K^2=6,$ $G=D_4 \times \mathbb{Z}/2\mathbb{Z}.$ Then $X$ is a Mori dream space.
\end{theorem}
\begin{proof}
From the above discussion we see that $K_X$ is numerically equivalent to $2E_1+2E_2+3F_1+G_1.$ Then using ramification formula we see that the pull-back of the anticanonical divisor of $X/Z$ is numerically equivalent to $2E_1+2E_2+2F_1+2G_1.$ Because we know $\Eff(X)$ we can check that it is nef and big divisor. Therefore we see that $X/Z$ is a Mori dream surface with Picard number 4. Therefore we see that $X$ is a Mori dream surface from the Proposition \ref{criterion}.
\end{proof}

\subsection{Keum-Naie surfaces : product-quotient surfaces with $K^2=4,$ $G=\mathbb{Z}/4\mathbb{Z} \times \mathbb{Z}/2\mathbb{Z}$ case}

Let $X$ be a product-quotient surface with $p_g=q=0, K^2=4$ and $G=\mathbb{Z}/4\mathbb{Z} \times \mathbb{Z}/2\mathbb{Z}.$ Note that these surfaces form a 2-dimensional subfamily of 6-dimensional Keum-Naie surfaces with $K^2=4.$ See \cite{BC2} for more details. An explicit description of such surfaces was provided in \cite{BCGP}. In this case, $C$ and $D$ are curves of genus 3 with $G$-action and the $G$-action is encoded in the following data $t_1 : (2,2,4,4),$ $S_1 : (2,1),(2,1),(3,1),(1,1)$ and $t_2 : (2,2,4,4),$ $S_2 : (0,1),(0,1),(3,0),(1,0).$ There are four singular points of type $4 \times \frac{1}{2}(1,1)$ on $(C \times D)/G$ and see \cite{BCGP} for more details. \\

From the results of Serrano and Polizzi and the above date we can compute the fibration structures of $X.$ Let $E_1,E_2,E_3,E_4$ be the 4 exceptional divisors in $X$ ordered counterclockwise and let $F_1$(resp. $F_2$) be the reduced fiber of one fibration to $\mathbb{P}^1$ meeting $E_1,E_2$(resp. $E_3,E_4$) and let $G_1$(resp. $G_2$) be the reduced fiber of another fibration to $\mathbb{P}^1$ meeting $E_2,E_3$(resp. $E_4,E_1$). We can compute intersections between these curves as follows. \\

\begin{center}
\begin{tabular}{|c|c|c|c|c|c|c|c|c|} \hline
$\cdot$ & $E_1$ & $E_2$ & $E_3$ & $E_4$ & $F_1$ & $F_2$ & $G_1$ & $G_2$ \\
\hline $E_1$ & -2 & 0 & 0 & 0 & 1 & 0 & 0 & 1 \\
\hline $E_2$ & 0 & -2 & 0 & 0 & 1 & 0 & 1 & 0 \\
\hline $E_3$ & 0 & 0 & -2 & 0 & 0 & 1 & 1 & 0 \\
\hline $E_4$ & 0 & 0 & 0 & -2 & 0 & 1 & 0 & 1 \\
\hline $F_1$ & 1 & 1 & 0 & 0 & -1 & 0 & 0 & 0 \\
\hline $F_2$ & 0 & 0 & 1 & 1 & 0 & -1 & 0 & 0 \\
\hline $G_1$ & 0 & 1 & 1 & 0 & 0 & 0 & -1 & 0 \\
\hline $G_2$ & 1 & 0 & 0 & 1 & 0 & 0 & 0 & -1 \\
\hline
\end{tabular}
\end{center}

From the above intersection matrix we can find a basis of $\Pic(X).$

\begin{lemma}
$F_1,E_1,G_2,E_4,F_2,E_3$ form a $\mathbb{Z}$-basis of $\Pic(X)/tors.$
\end{lemma}
\begin{proof}
One can compute the intersection matrix of $F_1,E_1,G_2,E_4,F_2,E_3.$
The determinant of the above matrix is $-1.$ Therefore we get the desired result.
\end{proof}

Now let us compute the canonical bundle.

\begin{lemma}
We have the following isomorphism.
$$ K_X \sim_{num} E_1+2G_2+2E_2+2F_2+E_3 $$
\end{lemma}
\begin{proof}
From the above lemma we have the following numerical equivalence relation.
$$ K_X \sim_{num} f_1F_1+e_1E_1+g_2G_2+e_4E_4+f_2F_2+e_3E_3 $$
Then we have $f_1=f_1+g_2-2e_1=g_2+f_2-2e_4=f_2-2e_3=0$ and $-f_1+e_1=e_1+e_4-g_2=e_3+e_4-f_2=e_3=1.$ Therefore we have $f_1=0, e_1=1, g_2=2, e_4=2, f_2=2, e_3=1.$
\end{proof}

The following numerical equivalences will play an important role.

\begin{lemma}
We have the following numerical equivalences.
$$ 2F_1+E_1+E_2 \sim_{num} 2F_2+E_3+E_4 $$
$$ 2G_1+E_2+E_3 \sim_{num} 2G_2+E_1+E_4 $$
$$ F_1+E_1+G_2 \sim_{num} F_2+E_3+G_1 $$
$$ F_1+E_2+G_1 \sim_{num} F_2+E_4+G_2 $$
\end{lemma}
\begin{proof}
One can directly check the above numerical equivalences via intersecting $F_1,E_1,G_2,E_4,F_2,E_3$ which form a $\mathbb{Z}$-basis of $\Pic(X)/tors.$
\end{proof}

We can compute the effective cone of $X$ as follows.

\begin{proposition}
Effective cone of $X$ is generated by $E_1,E_2,E_3,E_4,F_1,F_2,G_1,G_2.$
\end{proposition}
\begin{proof}
Let $D$ be an irreducible integral effective divisor on $X$ which lies on an extremal ray of $\Eff(X).$ Suppose that $D$ is does not lie on the convex cone generated by $E_1,E_2,E_3,E_4,F_1,F_2,G_1,G_2.$ Then intersection of $D$ with any divisor among the $E_1,E_2,E_3,E_4,F_1,F_2,G_1,G_2$ is nonnegative. Let us write $D \sim e_1E_1+e_2E_2+e_4E_4+f_1F_1+f_2F_2+g_2G_2. $ Then we have $f_1 \geq 0$ and $e_3 \geq 0.$ Intersecting $D$ with $F_1$ gives us $-f_1+e_1 \geq 0.$ Therefore we see that $e_1 \geq 0.$ Intersecting $D$ with $E_3$ gives us $-2e_3+f_2 \geq 0.$ Therefore we see that $f_2 \geq 0.$ Because $F_1+E_1+G_2$ is nef, we see that $e_4 \geq 0.$ Because $E_4+2F_2+E_3$ is nef, we see that $g_2 \geq 0.$ Then $D$ is a linear combination of $F_1,E_1,G_2,E_4,F_2,E_3$ with nonnegative coefficient which gives a contradiction. Therefore we get the desired result.
\end{proof}

Therefore we see that $\Eff(X)$ is a rational polyhedral cone generated by $E_1,E_2,E_3,E_4,$ $F_1,F_2,G_1,G_2.$ Then from the general facts of convex geometry we can find generators of $\Nef(X)$ explicitly. Let us recall the process to find the generators of $\Nef(X)$ described in \cite{Fulton2}. Let $A$ be a rational polyhedral cone in a real vector space $V$ of dimension $\rho.$ For every set of $\rho-1$ linearly independent vectors among generators of $A,$ consider a nonzero vector $w$ annihilating the set. If either $w$ or $-w$ is nonnegative for all generators of $A,$ then take $w$ as one of the generators of $A^\vee.$ From this process we can compute generators of $\Nef(X).$ \\

\begin{proposition}
The generators of $\Nef(X)$ are semiample and therefore $\Nef(X)=\SAmp(X).$
\end{proposition}
\begin{proof}
From the above process, we can describe all generators of $\Nef(X).$ The above proposition tells us that $\Eff(X)$ has eight extremal rays and hence we have 56 sets of five elements among the generators. Because the configuration of curves $E_1,E_2,E_3,E_4,F_1,F_2,G_1,G_2$ has symmetry (numerically), it is enough to check 10 configurations(it is easy to see this when we see the complement of the five elements from $\{ E_1,E_2,E_3,E_4,F_1,F_2,G_1,G_2 \}$) among the 56 sets. Let us check that $\Nef(X)=\SAmp(X)$ as follows. \\
(1) Consider the set $\{ F_1, E_1, G_2, E_4, F_2 \}.$ Then $F_1+E_1+G_2-F_2-E_3$ is a nonzero element which is orthogonal to all of them. However it is not multiple of a nef divisor since $(F_1+E_1+G_2-F_2-E_3) \cdot G_1 < 0$ and $(F_1+E_1+G_2-F_2-E_3) \cdot E_2 > 0.$ \\
(2) Consider the set $\{ E_1, G_2, E_4, F_2, E_3 \}.$ Then $-2F_1-E_1+E_4+2F_2+E_3$ is a nonzero element which is orthogonal to all of them. However it is not multiple of a nef divisor since $(-2F_1-E_1+E_4+2F_2+E_3) \cdot E_2 < 0$ and $(-2F_1-E_1+E_4+2F_2+E_3) \cdot G_1 > 0.$ \\
(3) Consider the set $\{ E_1, G_2, E_4, F_2, G_1 \}.$ Then $-F_1+G_2+E_4+F_2$ is a nonzero element which is orthogonal to all of them. However it is not multiple of a nef divisor since $(-F_1+G_2+E_4+F_2) \cdot E_2 < 0$ and $(-F_1+G_2+E_4+F_2) \cdot E_3 > 0.$ \\
(4) Consider the set $\{ E_1, G_2, E_4, E_3, G_1 \}.$ Then $E_1+2G_2+E_4$ is a nonzero element which is orthogonal to all of them. It is a nef divisor which is semiample since $E_1+2G_2+E_4 \sim E_2+2G_1+E_3.$ \\
(5) Consider the set $\{ G_2, E_4, F_2, G_1, E_2 \}.$ Then $G_2+E_4+F_2$ is a nonzero element which is orthogonal to all of them. It is a nef divisor which is semiample since $G_2+E_4+F_2 \sim G_1+E_2+F_1.$ \\
(6) Consider the set $\{ G_2, E_4, F_2, E_3, E_2 \}.$ Then $-E_1+E_4+2F_2+E_3$ is a nonzero element which is orthogonal to all of them. However it is not multiple of a nef divisor since $(-E_1+E_4+2F_2+E_3) \cdot F_1 < 0$ and $(-E_1+E_4+2F_2+E_3) \cdot G_1 > 0.$ \\
(7) Consider the set $\{ E_1, G_2, E_4, E_3, E_2 \}.$ Then $E_1+2G_2+E_4$ is a nonzero element which is orthogonal to all of them. It is a nef divisor which is semiample since $E_1+2G_2+E_4 \sim E_3+2G_1+E_2.$ \\
(8) Consider the set $\{ E_1, E_4, F_2, G_1, E_2 \}.$ Then $E_1+2G_2+2E_4+2F_2$ is a nonzero element which is orthogonal to all of them. It is a nef divisor which is semiample since $E_1+2G_2+2E_4+2F_2 \sim E_1+2F_1+2E_2+2G_1 \sim E_4+2F_2+E_3+2G_1+E_2.$ \\
(9) Consider the set $\{ E_1, G_2, F_2, G_1, E_2 \}.$ Then $E_1+2G_2+E_4+F_2$ is a nonzero element which is orthogonal to all of them. It is a nef divisor which is semiample since $E_1+2G_2+E_4+F_2 \sim F_2+E_3+2G_1+E_2 \sim E_1+G_2+G_1+E_2+F_1.$ \\
(10) Consider the set $\{ G_2, F_2, G_1, E_2, F_1 \}.$ Then $G_2+E_4+F_2$ is a nonzero element which is orthogonal to all of them. It is a nef divisor which is semiample since $G_2+E_4+F_2 \sim G_1+E_2+F_1.$ \\
Therefore we see that $\Nef(X)=\SAmp(X).$
\end{proof}

Therefore we have the following conclusion.

\begin{theorem}
$X$ is a Mori dream space.
\end{theorem}

\section{Negative curves and bounded negativity conjecture}

Bounded negativity conjecture is one of the oldest problems in the theory of algebraic surfaces. It is still a widely open problem. See \cite{BHKKMSRS} for more details about the conjecture.

\begin{conjecture}[Bounded negativity conjecture]
Let $X$ be a smooth projective surface. Then there is a nonnegative integer $b_X \geq 0$ such that for any negative curve $C$ the following inequality holds.
$$ C^2 \geq -b_X $$
\end{conjecture}

It seems to be well-known that bounded negativity conjecture is true for Mori dream surfaces among experts. Indeed, the proof is obvious.

\begin{proposition}[Bounded negativity conjecture]
Let $X$ be a smooth projective surface whose $\Eff(X)$ is a rational polyhedral cone. There are only finitely many negative curves on $X.$ In particular, the bounded negativity conjecture holds for Mori dream surfaces.
\end{proposition}
\begin{proof}
Let $C$ be a negative curve on $X.$ Then $C$ lies on one of the extremal rays of $\Eff(X)$ and it is the only irreducible reduced curve on the extremal ray. From our assumption $\Eff(X)$ is a rational polyhedral cone. Therefore there are only finitely many negative curves on $X.$
\end{proof}

Therefore the surfaces of general type with $p_g=0$ from Theorem \ref{Main} satisfy the bounded negativity conjecture. Moreover, from our explicit computation of $\Eff(X)$, we can describe all negative curves as in Theorem \ref{main}. Indeed, when $\rho(X)=1,$ we see that the self-intersection of a curve is always positive. When $X$ is a surfaces isogenous to a higher product of unmixed type, then we see that the intersection matrix is hyperbolic. When $\rho(X) \geq 3,$ $\Eff(X)$ is a convex polyhedral cone generated by negative curves since $\Eff(X)$ is a rational polyhedral cone. Conversely, every negative curve lies on the extremal rays of $\Eff(X).$ Therefore we got the desired results from the previous discussions.

\section{Discussions}

\subsection{Mimimal surfaces of general type with $p_g \neq 0$}

In this paper we have discussed Mori dream surfaces of general type with $p_g = 0.$ However, there is no reason to restrict one's attention to only those surfaces. Indeed, there are lots of Mori dream spaces with $p_g \neq 0.$ A simple example is as follows.

\begin{lemma}
A hypersurface $X$ in $\mathbb{P}^3$ is a Mori dream space if $\rho(X)=1.$
\end{lemma}
\begin{proof}
Let $X$ be a degree $d$ hypersurface in $\mathbb{P}^3.$ Then it is enough to prove that $q(X)=0.$ Consider the following short exact sequence.
$$ 0 \to \mathcal{O}_{\mathbb{P}^3}(-d) \to \mathcal{O}_{\mathbb{P}^3} \to \mathcal{O}_X \to 0 $$ We know that $H^1(\mathbb{P}^3,\mathcal{O}_{\mathbb{P}^3})=H^2(\mathbb{P}^3,\mathcal{O}_{\mathbb{P}^3}(-d))=0$ and hence $q=0.$ Because $\rho(X)=1, q=0$ we see that $X$ is a Mori dream space.
\end{proof}

From Noether-Lefschetz theorem, we have the following corollary.

\begin{corollary}
A very general hypersurface of degree $d \geq 4$ in $\mathbb{P}^3$ is a Mori dream space.
\end{corollary}

Indeed, we can find more examples as follows.

\begin{lemma}
A complete intersection variety of dimension greater than or equal to two with $\rho=1$ is a Mori dream space.
\end{lemma}
\begin{proof}
Let $X$ be a complete intersection variety, i.e. a zero locus of a regular section $s$ of $\mathcal{E}$ on $\mathbb{P}^N$ where $\mathcal{E}$ is a direct sum of ample line bundles. Then it is enough to prove that $q(X)=0.$ The Koszul resolution
$$ 0 \to \bigwedge^r\mathcal{E}^\vee \to \cdots \to \bigwedge^2\mathcal{E}^\vee \to \mathcal{E}^\vee \to \mathcal{O}_{\mathbb{P}^N} \to \mathcal{O}_X \to 0 $$
splits into short exact sequences
$$ 0 \to \mathcal{F}_0=I_X \to \mathcal{O}_{\mathbb{P}^N} \to \mathcal{O}_X \to 0 $$
$$ 0 \to \mathcal{F}_1 \to \mathcal{E}^\vee \to I_X \to 0 $$
$$ 0 \to \mathcal{F}_2 \to \bigwedge^2\mathcal{E}^\vee \to \mathcal{F}_1 \to 0 $$
$$ \cdots $$
$$ 0 \to 0=\mathcal{F}_r \to \bigwedge^r\mathcal{E}^\vee \to \mathcal{F}_{r-1} \to 0 $$
Then we have $H^1(X,\mathcal{O}_{X})=H^2(\mathbb{P}^N,I_X)=H^3(\mathbb{P}^N,\mathcal{F}_1)=\cdots=H^{r+1}(\mathbb{P}^N,\mathcal{F}_{r-1})=0,$ since $\mathcal{O}_{\mathbb{P}^N}, \mathcal{E}^\vee, \bigwedge^2\mathcal{E}^\vee, \cdots, \bigwedge^r\mathcal{E}^\vee$ are ACM bundles. Therefore we have $q=0,$ $\rho(X)=1$ and we see that $X$ is a Mori dream space.
\end{proof}

\begin{corollary}
A general complete intersection variety of dimension greater than or equal to two is a Mori dream space.
\end{corollary}

It will be an interesting task to have a systematic approach to study Mori dream spaces with $p_g \neq 0.$

\subsection{Numerical Godeaux surfaces and surfaces with $\kappa=1$}

When a minimal surface of general type has $p_g=q=0$ and $K^2=1$, we call this surface a numerical Godeaux surface. It is an interesting task to find examples of numerical Godeaux surfaces which are Mori dream surfaces. For example, it will be interesting to know whether the classical Godeaux surface is a Mori dream surface or not. It is also an interesting task to find an example of surface with $\kappa=1$ which is a Mori dream surface.

\subsection{Surfaces which are not Mori dream surfaces}

It is an interesting task to find examples of minimal surfaces of general type(especially those with $p_g=0$) which are not Mori dream surfaces. On the other hand, we can construct several examples of surfaces of general type which are not Mori dream surfaces via similar constructions used in Theorem \ref{notMori}.


\end{document}